\documentclass[12pt]{amsart} 

\usepackage{amsmath,amsthm,amssymb,amsfonts}
\usepackage{epsfig}
\usepackage{color}  
\usepackage{graphicx}

\newtheorem{thm}{Theorem}%[section]     
\newtheorem{lem}[thm]{Lemma}
\newtheorem{prop}[thm]{Proposition}
 
\theoremstyle{definition}

\theoremstyle{remark}
\newtheorem{remark}{Remark}%[section]
\newcommand{\tim}{\times}   
\newcommand{\R}{\mathbb R}
\newcommand{\T}{\mathbb T}
\newcommand{\Z}{\mathbb Z}
\newcommand{\pl}{\partial}

\newcommand{\N}{\mathbb N}

\renewcommand{\d}{\,\mathrm{d}}

\newcommand{\disp}{\displaystyle}

\newcommand{\gth}{\theta}

\newcommand{\gep}{\varepsilon}
\newcommand{\gd}{\delta}

\newcommand{\fr}{\frac}

\newcommand{\bye}{\end{document}}
\newcommand{\by}{\end{proof}\end{document}}

\def\gl{\lambda}

\def\ga{\alpha}
\def\gl{\lambda}

\def\go{\omega}

\def\Lip{\!\,{\rm Lip}\,}
\def\UC{\,{\rm UC}\,}

\def\mid{\,:\,}
\def\aln{&\,}
\def\erf{\eqref}
\begin{document}
%%%%%%%%%%%%%%%%%

\def\olQ{\,\overline{\!Q}} 
\def\olB{\,\overline{\!B}} 
\def\olD{\,\overline{\!D}}
\def\x{\hat x} \def\t{\hat t}\def\s{\hat s}
\def\p{\hat p}\def\q{\hat q}
\def\red#1{\textcolor{red}{#1}} 
\def\tny#1{\scriptscriptstyle#1}

\title[New PDE approach to the large time asymptotics]{ 
A new PDE approach to the large time asymptotics of solutions   
of Hamilton-Jacobi equations} 

\author[G. Barles, H. Ishii and H. Mitake] 
{Guy Barles, Hitoshi Ishii and Hiroyoshi Mitake}

\address[G. Barles]
{Laboratoire de Math\'ematiques et Physique 
Th\'eorique (UMR CNRS 7350), F\'ed\'eration Denis Poisson, 
Universit\'e de Tours (FR CNRS 2964), 
Place de Grandmont, 
37200 Tours, FRANCE}
\email{barles@lmpt.univ-tours.fr}
\urladdr{http://www.lmpt.univ-tours.fr/~barles} 

\address[H. Ishii]
{Faculty of Education and Integrated Arts and Sciences, Waseda University, Nishi-Waseda,
Shinjuku, Tokyo 169-8050, Japan/ Faculty of Science, King Abdulaziz University, 
P. O. Box 80203  Jeddah, 21589 Saudi Arabia.}
\email{hitoshi.ishii@waseda.jp}
\urladdr{http://www.edu.waseda.ac.jp/~ishii/}

\address[H. Mitake]
{Department of Applied Mathematics, 
Faculty of Science, Fukuoka University, Fukuoka 814-0180 Japan}
\email{mitake@math.sci.fukuoka-u.ac.jp}

\thanks{  
The work of HI was supported in part by KAKENHI  
\#20340019, \#21340032, \#21224001, \#23340028 and
\#23244015, JSPS}
\thanks{  
The work of HM was supported in part by KAKENHI  
\#24840042, JSPS and Grant for Basic Science Research Projects from the 
Sumitomo Foundation}

\dedicatory{Dedicated to Professor Neil S. Trudinger on the occasion of his 70th birthday}

%    General info

%\date{January 1, 2001 and, in revised form, June 22, 2001.}



\begin{abstract}
We introduce a new PDE approach to establishing 
the large time asymptotic behavior of solutions of Hamilton-Jacobi equations,  
which modifies and simplifies the previous ones (\cite{BS, BIM}), under 
a refined ``strict convexity'' assumption on the Hamiltonians.  
Not only 
such ``strict convexity'' conditions 
generalize the corresponding requirements on the 
Hamiltonians in \cite{BS}, but also one of the most refined our conditions covers 
the situation studied in \cite{NR}.     
\end{abstract}
\maketitle

\section{Introduction}  

In this article we introduce a new PDE approach to establishing 
the large time asymptotic behavior of solutions of Hamilton-Jacobi equations. 

In the last two decades there have been major developments in the study of 
the large time asymptotics of solutions of Hamilton-Jacobi equations, 
initiated by the work by Namah and Roquejoffre \cite{NR} and by Fathi \cite{F}. 

The approach by Fathi is based on the weak KAM theory and the representation of solutions 
of the Hopf-Lax-Oleinik type or, in other words, as the value functions of optimal control, 
and has a wide scope which is different from the one in Namah-Roquejoffre \cite{NR}. 
The optimal control/dynamical approach of Fathi has been subsequently developed for further applications 
and technical improvements by many authors (see, for instance, \cite{DS, FIL, II3, I1, I3, M1, M2}). 

At the beginning of the developments mentioned above, 
another approach has been introduced by the first author and Souganidis \cite{BS}, 
which does not 
depend on the representation formulas of solutions and thus applies to a more general class of 
Hamilton-Jacobi equations including those with non-convex Hamiltonians. 
We refer for recent developments in this direction to \cite{BM, BIM}. 

We also refer \cite{BIM} for further comments and references related to the large time asymptotics 
of solutions of Hamilton-Jacobi equations.    

Our aim here is to modify and slightly simplify the main ingredient in 
the PDE approach by the first author and Souganidis \cite{BS} 
as well as 
to refine the requirements on the Hamiltonians.

To clarify and simplify 
the presentation, we consider the asymptotic problem in the periodic setting. 
We are thus concerned with the Cauchy problem 
\[\left\{
\begin{aligned}
&u_t(x,t)+H(x,D_xu(x,t))=0&&\text{ in } Q,\\ 
&u(x,0)=u_0(x)&& \text{ for }\ x\in\R^n,
\end{aligned}\right.\tag{CP}
\]
where $Q:=\R^n\tim(0,\infty)$, $u$ represents the unknown function on $\olQ$, 
$u_t=u_t(x,t)=(\pl u/\pl t)(x,t)$, $D_xu(x,t)=((\pl u/\pl x_1)(x,t),...,(\pl u/\pl x_n)(x,t))$ 
and $u_0$ represents 
the initial data. The functions $u(x,t)$ and $u_0(x)$ are supposed to be periodic in $x$.  

We make the following assumptions throughout this article: 

\begin{itemize} 
\item[(A1)] The function $u_0$ is continuous in $\R^n$ and periodic with period  
$\Z^n$.  
\item[(A2)] $H\in C(\R^n\tim\R^n)$. 
\item[(A3)] The Hamiltonian $H(x,p)$ is periodic in $x$ with period $\Z^n$ for every $p\in\R^n$. 
\item[(A4)] The Hamiltonian $H$ is coercive. That is, 
\[
\lim_{r\to \infty}\inf\{H(x,p)\mid (x,p)\in\R^{2n},\, |p|\geq r\}=\infty. 
\]
\end{itemize}

Our notational conventions are as follows. 
We may regard functions $f(x)$ on $\R^n$ (resp., $g(x,y)$ on $\R^n\tim V$, where $V$ is a subset 
of $\R^m$) periodic in $x\in\R^n$ with period $\Z^n$ as functions on the torus $\T^n$ 
(resp., $\T^n\tim V$). In this viewpoint, we write $C(\T^n)$, $C(\T^n\tim V)$, etc, for  
the subspaces of all functions $f(x)$ in $C(\R^n)$, of all functions $g(x,y)$ 
in $C(\R^n\tim V)$, etc,  
periodic in $x$ with period $\Z^n$. We denote the sup-norm (or the $L^\infty$-norm) 
of a function $f$ by $\|f\|_\infty$ and $\|f\|_{L^\infty}$ interchangeably. 
Regarding the notion of solution of Hamilton-Jacobi 
equations, in this article we will be only concerned 
with viscosity solutions, viscosity subsolutions and viscosity supersolutions, 
which we refer simply as solutions, subsolutions and supersolutions.  
For any $R>0$, $B_R$ denotes the open ball of $\R^n$ with center at the origin 
and radius $R$.  
For any $X\subset\R^n$, $\UC(X)$ and $\Lip(X)$ denote the spaces of all uniformly continuous 
functions and all Lipschitz continuous functions on $X$, respectively.

We now recall the following basic results.  

\begin{thm} \label{thm:CP} Under the hypotheses \emph{(A1)--(A4)}, 
there exists a unique solution $u\in \UC(\T^n\tim[0,\,\infty))$ of \emph{(CP)}. 
Furthermore, if $u_0\in\Lip(\T^n)$, then $u\in\Lip(\T^n\tim[0,\,\infty))$.  
\end{thm} 

\begin{thm} \label{thm:comp} Under the hypotheses \emph{(A2)--(A4)}, 
let $u,\,v\in\UC(\T^n\tim[0,\,\infty))$ be solutions of 
\[
u_t+H(x,D_xu)=0 \ \ \ \text{ in } \ Q. \tag{HJ}
\] 
Then  
\[
\|u-v\|_{L^\infty(Q)}\leq \|u(\cdot,\,0)-v(\cdot,\,0)\|_{L^\infty(\R^n)}.  
\]
\end{thm}

\begin{thm} \label{thm:EP} 
Under the hypotheses \emph{(A2)--(A4)}, there exists a unique constant $c\in\R$ 
such that the problem
\[
H(x,Dv(x))=c \ \ \ \text{ in } \ \R^n \tag{EP}
\]
has a solution $v\in \Lip(\T^n)$. 
\end{thm}

These theorems are classical results in viscosity solutions theory. 
For instance, the existence part of Theorem \ref{thm:CP} is a consequence of 
Corollaire II.1 in \cite{B85}. Under assumptions (A2) and (A3), as is well known,
the comparison principle holds between bounded 
semicontinuous sub and supersolutions of (CP) if one of them is Lipschitz continuous. 
This comparison result and the existence part of Theorem \ref{thm:CP} assure that 
for each continuous solution $u$ of (CP) there is a sequence $\{u_k\}_{k\in\N}$ 
of Lipschitz continuous solutions of (CP), with $u_0$ replaced by $u_k(\cdot,0)$, 
which converges to $u$ 
uniformly in $\olQ$. The existence of such a sequence of Lipschitz 
continuous solutions of (CP) and the comparison principle for Lipschitz 
continuous 
solutions of (CP) guarantees the Theorem \ref{thm:comp} holds.   
Theorem \ref{thm:EP} and its proof can be found in \cite{LPV}.

The problem of finding a pair $(c,v)\in\R\tim C(\T^n)$, where $v$ satisfies (EP) 
in the viscosity sense, is called an additive eigenvalue problem or 
ergodic problem. Thus, for such a pair $(c,v)$, the function $v$ (resp., the constant $c$) 
is called an additive eigenfunction (resp., eigenvalue). 

We note that the conditions (A2)--(A4) are invariant under addition of constants. 
Hence, by replacing $H$ by $H-c$, with $c$ being the additive eigenvalue of (EP), 
we may normalize so that the additive eigenvalue $c$ is zero.  
Thus, in what follows, we always assume that 
\begin{itemize}
\item[(A5)] \ $c=0$, where $c$ denotes the additive eigenvalue.   
\end{itemize}
Accordingly, problem (EP) becomes simply a stationary problem 
\begin{equation}\label{S}
H(x,\,Dv(x))=0 \ \ \ \text{ in } \ \R^n.
\end{equation}

The crucial assumptions in this article are the following conditions. 

\begin{itemize}
\item[(A6)$_{\tny+}$]There exist constants $\eta_0>0$ and $\gth_0>1$ and for each 
$(\eta,\gth)\in (0,\eta_0)\tim(1,\gth_0)$ a constant $\psi=\psi(\eta,\gth)>0$ such that 
for all $x, p,q\in\R^n$, if $H(x,p)\leq 0$ and $H(x,q)\geq \eta$, then 
\[
H(x,p+\gth (q-p))\geq \eta\gth +\psi. 
\] 
\item[(A6)$_{\tny-}$]There exist constants $\eta_0>0$ and $\gth_0>1$ and for each 
$(\eta,\gth)\in (0,\eta_0)\tim(1,\gth_0)$ a constant $\psi=\psi(\eta,\gth)>0$ such that 
for all $x, p,q\in\R^n$, if $H(x,p)\leq 0$ and $H(x,q)\geq -\eta$, then 
\[
H(x,p+\gth (q-p))\geq -\eta\gth +\psi. 
\] 
\end{itemize}

We will furthermore modify and refine these conditions 
(see (A9)$_{\tny\pm}$) in Section 4, one of which covers 
the situation studied by Namah-Roquejoffre \cite{NR}. An important consequence 
is that our PDE method gives a unified approach to most of the large time asymptotic 
convergence results for (CP) in the literature.  

The assumptions above 
are some kind of strict convexity requirements and they are satisfied if $H$ is strictly convex in $p$. Indeed in this case, since $q = \theta^{-1}(p+\theta(q-p)) +(1-\theta^{-1})p$,
\begin{align*}
H(x,q) &< \theta^{-1}H(x, p+\theta(q-p)) + (1-\theta^{-1})H(x, p)\\
& <  \theta^{-1}H(x, p+\theta(q-p))\; ,
\end{align*}
and $\psi$ measures how strict is this inequality. We point out that, for (A6)$_{\tny-}$, this argument is valid if $p\neq q$ and the inequality is obvious if $p=q$, while in the case of (A6)$_{\tny+}$ clearly we have always $p\neq q$.

%In the case of (A6)$_{\tny+}$, 
One may have another interpretation of these assumptions, namely 
%says at least that,  
%for each fixed $x,p,q\in\R^n$, if $H(x,p)=0$, $H(x,q)=\eta>0$ and $\eta$ is sufficiently small, then  
that the function $H(x,r)$, as a function of $r$, grows more than linearly 
on the line segment connecting from $q$ to $p+\gth_0(q-p)$ for some $\gth_0>1$. 
(Notice that this growth rate is negative in the case of (A6)$_{\tny-}$.)

We conclude these remarks on (A6)$_{\tny\pm}$ by pointing out that (A6)$_{\tny+}$ is an assumption on the behavior of $H$ on the set $\{H\geq 0\}$ while (A6)$_{\tny-}$ is an assumption on the behavior of $H$ on the set $\{H\leq 0\}$. We refer to Section 3 for more precise comments in this direction.

A condition similar to (A6)$_{\tny+}$ has appeared first in Barles-Souganidis \cite{BS} 
(see (H4) in \cite{BS}). Our condition (A6)$_{\tny+}$ is less stringent and has 
a wider application than (A6)$_{\tny+}$ in \cite{BIM}. For this comparison, see Section 3. 
Also, (A6)$_{\tny-}$ is less stringent than (A6)$_{\tny-}$ in \cite{BIM}. 
A type of condition (A6)$_{\tny-}$ has first introduced in Ichihara-Ishii \cite{II2} for 
convex Hamiltonians (see the condition (16) in \cite{II2}).

We establish the following theorem by a PDE approach which modifies and simplifies 
the previous ones in \cite{BS, BIM}.   

\begin{thm}\label{thm:main} 
Assume that \emph{(A1)--(A5)} hold and that either \emph{(A6)$_{\tny+}$} or 
\emph{(A6)$_{\tny-}$} holds. 
Then the unique solution $u(x,t)$ in $\UC(\T^n\tim[0,\,\infty))$ of \emph{(CP)} 
converges uniformly in $\R^n$, as $t\to\infty$, 
to a function $u_\infty(x)$ in $\Lip(\T^n)$, which is a solution of \eqref{S}.
\end{thm}

A generalization of the theorem above is given in Section 4 (see Theorem \ref{thm:general}), 
which covers the main result in \cite{NR} in the periodic setting.

In Section 2, we give 
an explanation of the new ingredient in our new PDE method, 
a (hopefully transparent) formal proof of Theorem \ref{thm:main} 
by the new PDE method and  its exact version. 
In Section 3, we make comparisons between (A6)$_{\tny\pm}$ and its classical versions, 
and discuss convexity-like properties of the Hamiltonians $H$ implied by (A6)$_{\tny\pm}$ 
as well as a couple of conditions equivalent to (A6)$_{\tny\pm}$.  
In Section 4, we present a theorem, with (A6)$_{\tny\pm}$ replaced by refined 
conditions, which includes the situation in \cite{NR} as a special case.

\section{Proof of Theorem \ref{thm:main}} 

Throughout this section, we assume that (A1)--(A5) hold. The first step consists in reducing to the case when $u_0\in \Lip(\T^n)$ and therefore $u$ is Lipschitz continuous on $\T^n\tim[0,\,\infty)$.

\begin{lem}If the result of Theorem \ref{thm:main} holds for any $u_0 \in \Lip(\T^n)$ then it holds for any $u_0 \in C(\T^n)$.
\end{lem}

\begin{proof}

For a general $u_0\in C(\T^n)$ we select a sequence $\{u_{0,j}\}_{j\in\N}\subset\Lip(\T^n)$ 
which converges to $u_0$ uniformly in $\R^n$. For each $j\in\N$ let 
$u_j\in\Lip(\T^n\tim [0,\,\infty))$ be the unique solution of (CP), with $u_{0,j}$ 
in place of $u_0$. By Theorem \ref{thm:comp}, 
we have
\begin{equation}\label{2.1.3.6} 
\|u_j-u_k\|_{L^\infty(Q)}\leq \|u_{0,j}-u_{0,k}\|_{L^\infty(\R^n)} \ \ \ \text{ for all } 
\ j,k\in\N.
\end{equation}

Since Theorem \ref{thm:main} holds for any initial data in $\Lip(\T^n)$, we know that for each $j\in\N$ there exists a function 
$u_{\infty,j}\in C(\T^n)$ such that $\lim_{t\to\infty}u_j(x,t)=u_{\infty,j}(x)$ 
uniformly in $\R^n$. This implies 
\[
\|u_{\infty,j}-u_{\infty,k}\|_{L^\infty(\R^n)}\leq \|u_j-u_k\|_{L^\infty(Q)} 
\ \ \ \text{ for all } 
\ j,k\in\N,
\]
which together with \erf{2.1.3.6} 
yields 
\[
\|u_{\infty,j}-u_{\infty,k}\|_{\infty}\leq \|u_{0,j}-u_{0,k}\|_\infty 
\ \ \ \ \text{ for all } \ j,k\in\N.
\] 
Hence there is a function $u_\infty\in C(\T^n)$ such that 
$\lim_{j\to\infty}u_{\infty,j}(x)=u_\infty(x)$ uniformly in $\R^n$.

Observe by using Theorem \ref{thm:comp} that for any $j\in\N$, 
\begin{align*} 
\|u(\cdot,t)-u_\infty\|_\infty
&\leq \|u(\cdot,t)-u_j(\cdot,t)\|_\infty
+\|u_j(\cdot,t)-u_{\infty,j}\|_\infty
+\|u_{\infty,j}-u_\infty\|_\infty
\\&\leq \|u_0-u_{0,j}\|_\infty
+\|u_j(\cdot,t)-u_{\infty,j}\|_\infty
+\|u_{\infty,j}-u_\infty\|_\infty, 
\end{align*}
from which we conclude that $\lim_{t\to\infty}\|u(\cdot,t)-u_\infty\|_{\infty}=0$. 
By the stability property of viscosity solutions, we see that $u_\infty$ is a solution of 
\eqref{S} and, consequently, $u_\infty\in\Lip(\R^n)$ by Theorem \ref{thm:EP}.  
\end{proof}

Now we turn to the proof of Theorem \ref{thm:main} when $u_0 \in \Lip(\T^n)$. By Theorem \ref{thm:CP}, there exists a unique solution $u\in \Lip(\T^n\tim[0,\,\infty))$ of (CP) and we have to prove that 
$u(x,t)$ converges uniformly in $\R^n$ to a function $u_\infty(x)$ as $t\to\infty$. 

Henceforth in this section we assume that $u_0\in\Lip(\T^n)$ and hence the solution 
$u$ of (CP) is in $\Lip(\T^n\tim[0,\,\infty))$.  Also, 
we fix a solution 
$v_0\in \Lip(\T^n)$ of \eqref{S}. Such a function $v_0$ 
exists thanks to Theorem \ref{thm:EP}.  
We set $L:=\max\{\|D_x u\|_\infty, \|D_x v_0\|_\infty\}$.

If we set $z(x,t)=v_0(x)$ and invoke Theorem \ref{thm:comp}, then we get
\[
\|u-z\|_{L^\infty(Q)}\leq \|u_0-v_0\|_{L^\infty(\R^n)},
\] 
which shows that $u$ is bounded in $\olQ$.   
We may assume by adding a constant to $v_0$ if needed that for some constant $C_0>0$,  
\[
0\leq u(x,t)-v_0(x)\leq C_0 \ \ \ \text{ for all } \ (x,t)\in\olQ.
\]

\subsection{Under assumption (A6)$_{\tny+}$} 
Throughout this subsection we assume, in addition to (A1)--(A5), that (A6)$_{\tny+}$ holds. 
Let $\eta_0>0$ and $\gth_0>1$ be the constants from (A6)$_{\tny+}$.

For $(\eta,\theta) \in (0,\eta_0)\tim (1,\theta_0)$, we define the function $w$ on $\olQ$ by
\begin{equation}\label{eq2}
w(x,t)=\sup_{s\geq t}[u(x,t)-v_0(x)-\gth(u(x,s)-v_0(x)+\eta(s-t))].  
\end{equation}
%Here we are assuming that $0\leq u(x,t)-v(x)\leq C$ for all 
%$(x,t)\in\R^n\tim[0,\,\infty)$.  

The following proposition is crucial in our proof of Theorem \ref{thm:main} 
under (A6)$_{\tny+}$.  
To state the proposition, we introduce the functions 
$\go_{H,R}$, with $R>0$, as  
\[
\go_{H,R}(r)= \sup\{|H(x,p)-H(x,q)|\mid 
x\in\R^n,\, p,q\in \olB_R,\, |p-q|\leq r\}. 
\]
Note that for each $R>0$, the function $\go_{H,R}$ 
is nonnegative and nondecreasing in $[0,\,\infty)$ 
and $\go_{H,R}(0)=0$.

%Assume that $u_0\in\Lip(\R^n)$.  
%Let $0<\eta<\eta_0$ and $1<\gth<\gth_0$, where $\eta_0$ and $\gth_0$ are from \emph{(A6)$_{\tny +}$}. 

\begin{prop}\label{prop:main} 
Let $\psi=\psi(\eta,\gth)>0$ be the constant from \emph{(A6)$_{\tny+}$}.  
Then the function $w$ is a subsolution of 
\begin{equation}\label{eq3}
\min\{w(x,t),w_t(x,t)-\go_{H,R}(|D_x w(x,t)|)+\psi\}\leq 0 \ \ \ \text{ in } \ 
Q, 
\end{equation} 
where $R:=(2\theta_0+1)L$.  
\end{prop}

Our proof of Theorem \ref{thm:main} follows the outline of previous works like \cite{BS,BIM} where a key result is an asymptotic monotonicity property for $u$. This asymptotic monotonicity is a consequence of Proposition~\ref{prop:main} which, roughly speaking, implies that 
$\min\{u_t,0\} \to 0$ as $t \to \infty$. This is rigourously stated in Lemma~\ref{lem:w-infty} and its consequence in \erf{asymp.mon}. With assumption (A6)$_{\tny-}$, this is also the case but with a different monotonicity (i.e., $\max\{u_t,0\} \to 0$ as $t \to \infty$).

For this reason, the function $w$ defined by \erf{eq2} is a kind of Lyapunov function 
in our asymptotic analysis in a broad sense.  
The main new aspect in this article, compared to \cite{BS,BIM}, is indeed the simpler form of our $w$, which is defined 
by taking supremum in $s$ of the function 
\[
u(x,t)-v_0(x)-\gth(u(x,s)-v_0(x)+\eta(s-t)),
\]
whose functional dependence on $u$ and $v_0$ is linear. 
In the previous works, the function 
\begin{equation}\label{eq4} 
\sup_{s\geq t}\fr{u(x,s)-v_0(x)+\eta(s-t)}{u(x,t)-v_0(x)} 
\end{equation}
(one should assume here by adding a constant to $v_0$ if necessary 
that $\inf_{(x,t)\in Q}(u(x,t)-v_0(x))>0$), 
played the same role as our function $w$, and the value 
\[
\fr{u(x,s)-v_0(x)+\eta(s-t)}{u(x,t)-v_0(x)}
\]
depends nonlinearly in $u$ and $v_0$. 
One might see that the passage from the function given by \erf{eq4} 
to $w$ given by \erf{eq2} bears 
a resemblance that from the Kruzkov transform to 
a linear change in \cite{I} in the analysis of the comparison principle 
for stationary Hamilton-Jacobi equations. 

From a technical point of view, they are a lot of variants for such results. For example, as it is the case in \cite{BS}, one may look for a variational inequality for $m(t):= \max_{x\in\R^n}w(x,t)$ or for $m(t):= \max_{x\in\overline \Omega}w(x,t)$ where $\Omega$ is a suitable domain of $\R^n$. This last form can be typically useful when one wants to couple different assumptions on $H$ on $\Omega$ and its complementary as in \cite{BS} where the coupling with Namah-Roquejoffre type assumptions was solved in that way, the point being to control the behavior of $u$ on $\partial \Omega$.

For the connections between our assumptions and Namah-Roquejoffre type assumptions, we refer to Section~\ref{sec:NR}.

\subsubsection{A formal computation} Here we explain the algebra which bridges 
condition (A6)$_{\tny+}$ to Proposition \ref{prop:main} under the strong  
regularity assumptions that $u,\,w\in C^1(\T^n\tim[0,\,\infty))$ and $v_0\in C^1(\T^n)$ and that for each 
$(x,t)\in Q$ there exists an $s>t$ such that 
\begin{equation}\label{eq5}
w(x,t)=u(x,t)-v_0(x)-\gth(u(x,s)-v_0(x)+\eta(s-t)). 
\end{equation}
Of course, these conditions do not hold in general. 

Fix any $(x,t)\in Q$ and an $s>t$ so that \erf{eq5} holds.  
If $w(x,t)\leq 0$, then \erf{eq3} holds at $(x,t)$.  
We thus suppose that $w(x,t)>0$.   

Setting 
\[
p=Dv_0(x),\quad
q=D_xu(x,s), \quad
r=D_xu(x,t),\quad
a=u_s(x,s) \quad\text{and} \quad 
b=u_t(x,t),
\]
we have
\begin{align}
H(x,p)&\leq 0.  \label{eq6}
\\a+H(x,q)&\geq 0, \label{eq7}
\\b+H(x,r)&\leq 0, \label{eq8} 
\end{align}
Also, by the choice of $s$, we get
\begin{align}
D_xw(x,t)&=r-p-\gth(q-p), \label{eq9}
\\w_t(x,t)&=b+\gth\eta, \label{eq10}
\\0&=-\gth(a+\eta).\label{eq11}
\end{align}
Combining \eqref{eq7} and \eqref{eq11} yields
\begin{equation}\label{eq12}
H(x,q)\geq \eta.
\end{equation}

Now, in view of 
inequalities \eqref{eq6} and \eqref{eq12}, we may use assumption (A6)$_{\tny+}$, to get 
\[
H(x,p+\gth(q-p))\geq \gth\eta+\psi. 
\] 
Using \eqref{eq9}, we get
\[
H(x,r)=H(x,D_xw(x,t)+p+\gth(q-p)).
\] 
Using the definition of $L>0$, we clearly have
$$
|r|=|D_x u(x,t)|\leq L \leq R\; , \; |p+\gth(q-p)|\leq (1+2\theta) L \leq R
$$
and therefore we get 
\begin{align*}
H(x,r)
&\geq H(x,p+\gth(q-p))-\go_{H,R}(|D_xw(x,t)|)
\\&\geq 
-\go_{H,R}(|D_xw(x,t)|)+\gth\eta+\psi.
\end{align*}
This together with \eqref{eq8} and \eqref{eq10} yields
\begin{align*}
0&\geq b+H(x,r)=w_t(x,t)-\gth\eta +H(x,r)
\\&\geq w_t(x,t)-\gth\eta-\go_{H,R}(|D_xw(x,t)|)
+\gth\eta+\psi
\\&=w_t(x,t)-\go_{H,R}(|D_xw(x,t)|)+\psi.  
\end{align*}
This shows under our convenient regularity assumptions 
that \erf{eq3} holds. 

\begin{remark}\label{rem:v_0} The actual requirement to $v_0$ is just the subsolution property in the above 
computation, which is true also in the following proof of Theorem \ref{thm:main}. 
Some of subsolutions of \eqref{S} may have a better property, which solutions of \eqref{S} 
do not have. This is the technical insight in the generalization of 
Theorem \ref{thm:main} in Section 4. 
\end{remark}

\subsubsection{Proof of Proposition \ref{prop:main}} 
We begin with the following lemma.

\begin{lem} \label{lem-bounds} We have
\[
-C_0(\gth-1)\leq w(x,t)\leq C_0 \ \ \ \text{ for all } \ (x,t)\in\R^n\tim[0,\,\infty).
\]
\end{lem}

\begin{proof} We just need to note that for all $(x,t)\in\R^n\tim[0,\,\infty)$, 
\[
w(x,t)\geq u(x,t)-v_0(x)-\gth(u(x,t)-v_0(x))=(1-\gth)(u(x,t)-v_0(x))\geq -C_0(\gth-1), 
\]
and 
\[
w(x,t)\leq\max_{s\geq t}(u(x,t)-v(x))\leq C_0. \qedhere
\]
\end{proof}

\begin{proof}[Proof of Proposition \ref{prop:main}] 
Noting that $u\in\Lip(\T^n\tim[0,\,\infty))$ and $v_0\in\Lip(\T^n)$ 
%, according to Theorem \ref{thm:CP}, the solution $u$ is Lipschitz continuous 
%in $\olQ$. 
%Hence, 
and rewriting $w$ as   
\[
w(x,t)=\max_{r\geq 0}(u(x,t)-v_0(x)-\gth(u(x,r+t)-v_0(x)+\eta r)), 
\]
we deduce that $w\in \Lip(\T^n\tim[0,\,\infty))$.

Fix any $\phi_0\in C^1(Q)$ and $(\x,\t)\in Q$, and assume that 
\[
\max_{Q}(w-\phi_0)=(w-\phi_0)(\x,\t\,).
\] 
We intend to prove that for $R=(2\gth_0+1)L$,   
\begin{equation}\label{2.1.2.1}
\min\{w, \phi_{0,t}-\go_{H,R}(|D\phi_0|)+\psi\}\leq 0 \ \ \ \text{ at } \ (\x,\t). 
\end{equation}

%some constant $R>0$, which depends only on the Lipschitz constants of $w$, $u$ and $v_0$ 
%and the constant $\gth$. 

If $w(\x,\t)\leq 0$, then \erf{2.1.2.1} clearly holds.   
We may thus suppose that $w(\x,\t)>0$. We choose an $\s\geq\t$ so that
\[
w(\x,\t)=u(\x,\t)-v_0(\x)-\gth(u(\x,\s)-v_0(\x)+\eta(\s-\t)). 
\]
Observe that for any $s=\t$,  
\[
u(\x,\t)-v_0(\x)-\gth(u(\x,s)-v_0(\x)+\eta(s-\t)
=(1-\gth)(u(\x,\t)-v_0(\x))\leq 0, 
\] 
which guarantees that $\s>\t$.

Define the function $\phi\in C^1(Q\tim(0,\,\infty))$ by
\[
\phi(x,t,s)=\phi_0(x,t)+|x-\x|^2+(t-\t)^2+(s-\s)^2. 
\]
Note that the function 
\[
u(x,t)-v_0(x)-\gth(u(x,s)-v_0(x)+\eta(s-t))-\phi(x,t,s) 
\]
on $Q\tim(0,\,\infty)$ attains a strict maximum at $(\x,\t,\s)$, 
and that $D_x\phi(\x,\t,\s)=D_x\phi_0(\x,\t)$, $\phi_t(\x,\t,\s)=\phi_{0,t}(\x,\t)$ 
and $\phi_s(\x,\t,\s)=0$. 

Now, if $B$ is an open ball of $\R^{3n+2}$ centered at $(\x,\x,\x,\t,\s)$ 
with its closure $\olB$ contained 
in $\R^{3n}\tim(0,\,\infty)^2$, we use the technique of ``tripling variables'' 
and consider the function $\Phi$ on $\olB$ given by 
\begin{align*}
\Phi(x,y,z,t,s)=&
u(x,t)-v_0(z)-\gth(u(y,s)-v_0(z)+\eta(s-t))
\\&-\phi(x,t,s)-\ga(|x-y|^2+|x-z|^2),  
\end{align*}
where $\ga>0$ is a (large) constant.

Let $(x_\ga,y_\ga,z_\ga,t_\ga,s_\ga)\in\olB$ be a maximum point of $\Phi$. 
As usual in viscosity solutions theory, we observe that
\[
\lim_{\ga\to\infty}(x_\ga,y_\ga,z_\ga,t_\ga,s_\ga)=(\x,\x,\x,\t,\s). 
\]
Consequently, if $\ga$ is sufficiently large, then 
\[
(x_\ga,y_\ga,z_\ga,t_\ga,s_\ga)\in B. 
\]
We assume henceforth 
that $\ga$ is sufficiently large so that the above inclusion holds. 

Next, setting 
\[
p_\ga=2(\gth-1)^{-1}\ga(z_\ga-x_\ga) \ \ \ \text{ and } \ \ \ 
q_\ga=2\gth^{-1}\ga(x_\ga-y_\ga), 
\]
and noting that 
\begin{align*}
\Phi(x,y,z,t,s)=&
u(x,t)-\gth u(y,s)+(\gth-1)v_0(z)-\gth\eta(s-t)
\\&-\phi(x,t,s)-\ga(|x-y|^2+|x-z|^2), 
\end{align*}
we observe that  
\begin{align}
p_\ga&\in D^+v_0(z_\ga),\label{2.1.2.2}\\
\left(q_\ga,\,-\gth^{-1}\phi_s(x_\ga,t_\ga,s_\ga)-\eta\right)&\in D^-u(y_\ga,s_\ga),\label{2.1.2.3}\\
\left(D_x\phi(x_\ga,t_\ga,s_\ga)+\gth q_\ga-(\gth-1)p_\ga,\, 
\phi_t(x_\ga,t_\ga,s_\ga)-\gth\eta\right)&\in D^+u(x_\ga,t_\ga),\label{2.1.2.4}
\end{align}
By the definition of $L$, we see as usual in viscosity solutions theory that \ 
$\max\{|p_\ga|,\,|q_\ga|\}\leq L$. 
Sending $\ga\to\infty$ in \eqref{2.1.2.2}--\eqref{2.1.2.4} 
along an appropriate sequence, we find points 
$\p,\,\q\in \olB_L$ such that 
\begin{align}
\p&\in \olD{}^+v_0(\x),\label{2.1.2.5}\\
\left(\q,\,-\gth^{-1}\phi_s(\x,\t,\s)-\eta\right)&\in \olD{}^-u(\x,\s),\label{2.1.2.6}\\
\left(D_x\phi(\x,\t,\s)+\gth \q-(\gth-1)\p,\, 
\phi_t(\x,\t,\s)-\gth\eta\right)&\in \olD{}^+u(\x,\t),\label{2.1.2.7}
\end{align}
where $\olD{}^\pm$ stand for the closures of $D^\pm$, for instance, 
$\olD^+u(\x,\s)$ denotes the set of points $(q,b)\in\R^n\tim\R$ for which there 
are sequences $\{(q_j,b_j)\}_{j}\subset\R^n\tim\R$ and $\{(x_j,s_j)\}_{j}\subset Q$ such that 
$\lim_{j}(q_j,b_j,x_j,s_j)=(q,b,\x,\s)$ and $(q_j,b_j)\in D^+u(x_j,s_j)$ for all $j$. 
Here recall that $\phi_s(\x,\t,\s)=0$, 
$\phi_t(\x,\t,\s)=\phi_{0,t}(\x,\t)$
and $D_x\phi(\x,\t,\s)=D_x\phi_0(\x,\t)$.

From \eqref{2.1.2.5} and \eqref{2.1.2.6}, we get \ 
$H(\x,\p\,)\leq 0$ \ and 
\[
-\eta+H(\x,\q\,)\geq 0. 
\]
By condition (A6)$_{\tny+}$, we get
\begin{equation}\label{2.1.2.8} 
H(\x,\p+\gth(\q-\p\,))\geq \gth\eta+\psi. 
\end{equation}
From \eqref{2.1.2.7}, we get 
\begin{equation}\label{2.1.2.9}
0\geq \phi_{0,t}(\x,\t)-\gth\eta+H(\x,D_x\phi_0(\x,\t)+\gth\q-(\gth-1)\p\,). 
\end{equation}
Noting that \  
$|\hat p+\gth(\hat q-\hat p)|\leq (1+2\gth)L\leq R$ \  
and $|D_x\phi_0(\x,\t)+\gth \q-(\gth-1)\p|\leq L$ because of (\ref{2.1.2.7})
and combining \eqref{2.1.2.9} and \eqref{2.1.2.8}, we get 
\begin{align*}
0&\geq \phi_{0,t}(\x,\t\,)-\gth\eta+H(\x,\p+\gth (\q-\p))-\go_{H,R}(|D_x\phi_0(\x,\t\,)|)
\\& \geq\phi_{0,t}(\x,\t\,)-\go_{H,R}(|D_x\phi_0(\x,\t\,)|)+\psi,
\end{align*}
which shows that \eqref{2.1.2.1} holds. 
\end{proof}

\subsubsection{Completion of the proof of Theorem \ref{thm:main} under (A6)$_{\tny+}$} 
We set
\[
w_\infty(x)=\limsup_{t\to\infty}w(x,t) \ \ \ \text{ for all } \ x\in\R^n. 
\]

\begin{lem} \label{lem:w-infty} We have
\[
w_\infty(x)\leq 0 \ \ \ \text{ for all } \ x\in\R^n.
\]
Moreover, the convergence 
\begin{equation}\label{2.1.3.1}
\lim_{t\to\infty}\max\{w(x,t),\,0\}=0  
\end{equation}
is uniform in $x\in\R^n$.   
\end{lem}

\begin{proof} It is sufficient to prove that the convergence \eqref{2.1.3.1} holds uniformly 
in $x\in\R^n$.  Contrary to this, we suppose that there is a 
sequence $(x_j,t_j)\in Q$ such that $\lim_{j\to\infty}t_j=\infty$ and 
$w(x_j,t_j)\geq \gd$ for all $j\in\N$ and some constant $\gd>0$. In view of the periodicity 
of $w$, we may assume that $\lim_{j\to\infty}x_j=y$ for some $y\in\R^n$.  
Moreover, in view of the Ascoli-Arzela theorem, 
we may assume by passing to a subsequence of $\{(x_j,t_j)\}$
if needed that 
\begin{align*}
&\lim_{j\to\infty} w(x,t+t_j)=f(x,t) \ \ \ \text{ locally uniformly in } \ \R^n\times (-\infty,+\infty), 
\end{align*}
for some bounded function $f\in\Lip(\T^n\tim\R)$. 

Now, note that $f(y,0)\geq \gd$. By the stability of the subsolution property 
under uniform 
convergence, we see that $f$ is a subsolution of 
\begin{equation}\label{2.1.3.2}
\min\{f(x,t),\,f_t(x,t)-\go_{H,R}(|D_x f(x,t)|)+\psi\}\leq 0 \ \ \ \text{ in } \ 
\R^{n+1}. 
\end{equation} 
Since $f\in C(\T^n\tim\R)$ and $f$ is bounded on $\R^{n+1}$, for every  
$\gep>0$ the function $f(x,t)-\gep t^2$ attains a maximum over $\R^{n+1}$ 
at a point $(x_\gep,t_\gep)$. Observe as usual in the viscosity solutions theory 
that 
\[
f(x_\gep,t_\gep)-\gep t_\gep^2 \geq f(y,0)\geq\gd,
\]
and therefore 
\[
f(x_\gep,t_\gep)\geq\gd \ \ \ \text{ and } \ \ \ \gep |t_\gep|\leq (\gep\|f\|_\infty)^{1/2}. 
\] 
In particular, we have $\lim_{\gep\to 0+}\gep t_\gep=0$. 
In view of inequality \erf{2.1.3.2}, we get 
\[
2\gep t_\gep-\go_{H,R}(0)+\psi\leq 0, 
\]
which, in the limit as $\gep\to 0+$, yields $\psi\leq 0$, a contradiction. 
This shows that the uniform convergence \erf{2.1.3.1} holds.    
\end{proof}

\begin{proof}[Proof of Theorem \ref{thm:main} under (A6)$_{\tny+}$]  
Let $w$ be the function defined by \erf{eq2}, 
with arbitary $(\eta,\gth)\in(0,\eta_0)\tim(0,\gth_0)$. 

Fix any $\gep>0$. Thanks to \erf{2.1.3.1}, 
we may choose a constant $T_\gep\equiv T_{\gep,\eta,\gth}>0$ so that
for any $t\geq T_\gep$, 
\[
w(x,t)\leq \gep \ \ \ \text{ for all } \ x\in\R^n. 
\]
Let $t\geq T_\gep$ and $x\in\R^n$. From the above, for any $s\geq t$, we have
\begin{align*}
u(x,t)-v_0(x)\leq\aln \gep+\gth(u(x,s)-v_0(x))+\gth\eta(s-t)\\
=\aln\gep+u(x,s)-v_0(x)+(\gth-1)(u(x,s)-v_0(x))+\gth\eta(s-t)\\
\leq\aln \gep+u(x,s)-v_0(x)+(\gth-1)C_0+\gth\eta(s-t).  
\end{align*}
Thus, for any $0\leq s\leq 1$, we have  
\begin{equation}\label{2.1.3.3}
u(x,t)\leq u(x,t+s)+(\gth-1)C_0+\gth\eta+\gep. 
\end{equation} 

Now, since $u$ is bounded and Lipschitz continuous in $\olQ$, in view of the Ascoli-Arzela  theorem, 
we may choose a sequence $\tau_j\to\infty$ and a bounded 
function $z\in\Lip(\T^n \times (-\infty,+\infty))$ so that
\begin{equation}\label{2.1.3.4}
\lim_{j\to\infty} u(x,t+\tau_j)=z(x,t) \ \ \ \text{ locally uniformly on } \ \R^{n+1}.  
\end{equation}
By \erf{2.1.3.3} we get 
\[
z(x,t)\leq z(x,t+s) +(\gth-1)C_0+\gth\eta+\gep \ \ \ \text{ for all } \
(x,t,s)\in\R^{n+1}\tim[0,\,1]. 
\]
This is valid for all $(\eta,\gth)\in(0,\eta_0)\tim(1,\gth_0)$. Hence,  
\[
z(x,t)\leq z(x,t+s)+\gep \ \ \ \text{ for all } \ (x,t,s)\in\R^{n+1}\tim[0,\,1],
\]
and moreover 
\begin{equation}\label{asymp.mon}
z(x,t)\leq z(x,t+s) \ \ \ \text{ for all } \ (x,t,s)\in\R^{n+1}\tim[0,\,1].  
\end{equation}
Thus we find that the function $z(x,t)$ is nondecreasing in $t\in \R$ for all $x\in\R^n$. 
From this, we conclude that 
\begin{equation}\label{2.1.3.5}
\lim_{t\to\infty}z(x,t)=u_\infty(x) \ \ \ \text{ uniformly on } \ \R^n 
\end{equation}
for some function $u_\infty\in\Lip(\T^n)$.

Fix any $\gd>0$. By \erf{2.1.3.5} there is a constant $\tau>0$ such that  
\[
\|z(\cdot,\tau)-u_\infty\|_{L^\infty(\R^n)}<\gd,
\]
Then, by \erf{2.1.3.4} there is a $j\in\N$ such that 
\[
\|z(\cdot,\tau)-u(\cdot,\tau+\tau_j)\|_\infty<\gd. 
\]
Hence,
\[
\|u(\cdot,\tau+\tau_j)-u_\infty\|_\infty<2\gd. 
\]
By the contraction property (Theorem \ref{thm:comp}),  
we see that for any $t\geq\tau+\tau_j$, 
\[
\|u(\cdot,t)-u_\infty\|_\infty\leq \|u(\cdot,\tau+\tau_j)-u_\infty\|_\infty<2\gd,
\]
which completes the proof.  
\end{proof}

\subsection{Under assumption (A6)$_{\tny-}$}

In addition to (A1)--(A5), we assume
throughout this subsection that (A6)$_{\tny-}$ holds. 
 
To accommodate the previous $w$ to (A6)$_{\tny-}$, we modify and replace it by the new 
function, which we denote by the same symbol, given by  
\[
w(x,t)=\max_{0\leq s\leq t}(u(x,t)-v_0(x)-\gth(u(x,s)-v_0(x)-\eta(s-t)),  
\]
where $(\eta,\gth)$ is chosen arbitrarily in $(0,\,\eta_0)\tim(1,\,\gth_0)$ and 
the constants $\eta_0$ and $\gth_0$ are those from (A6)$_{\tny-}$. 

\begin{lem} \label{w-bound} We have
\[
-C_0(\gth-1)\leq w(x,t)\leq C_0 \ \ \ \text{ for all } \ (x,t)\in \olQ.
\]
\end{lem}

\begin{proof} Recall that $0\leq u(x,t)-v_0(x)\leq C_0$ for all 
$(x,t)\in\olQ$, and note that for all $(x,t)\in\olQ$, 
\[
w(x,t)\geq u(x,t)-v_0(x)-\gth(u(x,t)-v_0(x))=(1-\gth)(u(x,t)-v_0(x))\geq -C_0(\gth-1) 
\]
and 
\[
w(x,t)\leq\max_{
0\leq s\leq t}
(u(x,t)-v_0(x))\leq C_0. \qedhere
\]
\end{proof}

We have the following proposition similar to Proposition \ref{prop:main}.

\begin{prop}\label{prop:minus}  
The function $w$ is a subsolution of 
\begin{equation}\label{2.2.1}
\min\{w(x,t),w_t(x,t)-\go_{H,R}(|D_x w(x,t)|)+\psi\}\leq 0 \ \ \ \text{ in } \ 
(x,t)\in\R^n\tim(T,\,\infty), 
\end{equation}
where $\psi=\psi(\gth,\eta)>0$ is the constant from \emph{(A6)$_{\tny-}$}, 
$T:=C_0/\eta$ and $R:=(2\gth_0+1)L$.  
\end{prop}

Since the proof of the above proposition is very similar to that of 
Proposition \ref{prop:main}, we present just an outline of it. 

\begin{proof}[Outline of proof] Note that for any $(x,t)\in\R^n\tim(T,\,\infty)$ and 
$s\in[0,\,t-T)$, 
\begin{equation}\label{2.2.2}  
\begin{aligned}
u(x,t)-v_0(x)-\gth(u(x,s)-v_0(x)-\eta(s-t))
&\leq C_0-\gth\eta(t-s)
\\&<C_0-\gth\eta T=-(\gth-1)C_0. 
\end{aligned}
\end{equation}
Hence, in view of Lemma \ref{w-bound}, for any $(x,t)\in\R^n\tim(T,\,\infty)$ we have 
\begin{align*}
w(x,t)&=\max_{t-T\leq s\leq t}[u(x,t)-v_0(x)-\gth(u(x,s)-v_0(x)-\eta(s-t))]
\\&=\max_{-T\leq s\leq 0}[u(x,t)-v_0(x)-\gth(u(x,s+t)-v_0(x)-\eta s)]. 
\end{align*}
From this latter expression of $w$, 
as the functions $u$ and $v_0$ are Lipschitz continuous in $\olQ$ and $\R^n$, respectively, 
we see that $w$ is Lipschitz continuous in $\R^n\tim[T,\,\infty)$. 
Also, from \erf{2.2.2} we see that for any $(x,t)\in\R^n\tim(T,\,\infty)$, if 
\[
w(x,t)=u(x,t)-v_0(x)-\gth(u(x,s)-v_0(x)-\eta(s-t))
\]
for some $0\leq s\leq t$, then $s\geq t-T>0$. 

To see that \erf{2.2.1} holds, we fix any test function $\phi_0\in C^1(\R^n\tim(T,\,\infty))$ 
and assume that $w-\phi_0$ attains a strict maximum at a point $(\x,\t)$. 

Following the same arguments as in the proof 
under (A6)$_{\tny+}$, we are led to the inclusions 
\begin{equation}\label{2.2.4}\left\{
\begin{aligned}  
&\p\in \olD{}^+v_0(\x),
\\&\left(\q,\,\eta\right)\in \olD{}^-u(\x,\s),
\\&\left(D_x\phi_0(\x,\t)+\gth \q-(\gth-1)\p,\, 
\phi_{0,t}(\x,\t)+\gth\eta\right)\in \olD{}^+u(\x,\t)
\end{aligned}\right.
\end{equation}
for some $\hat p,\hat q\in\R^n$.

Using \eqref{2.2.4}, we observe that \ 
$H(\x,\p\,)\leq 0$ \ and \ $\eta+H(\x,\q\,)\geq 0$. 
Hence, by condition (A6)$_{\tny-}$, we get
\[ 
H(\x,\p+\gth(\q-\p\,))\geq -\gth\eta+\psi. 
\]
Moreover, we compute that  
\begin{align*}
0&\geq \phi_{t}(\x,\t,\s)+\gth\eta+H(\x,D_x\phi(\x,\t,\s)+\gth\q-(\gth-1)\p\,)
\\&\geq \phi_{0,t}(\x,\t)+\gth\eta
-\go_{H,R}(|D_x\phi_0(\x,\t)|)
+H(\x,\gth\q-(\gth-1)\p\,)
\\&\geq \phi_{0,t}(\x,\t)
-\go_{H,R}(|D_x\phi_0(\x,\t)|)
+\psi.
\end{align*}
Note that, as above, 
$|\p+\gth(\q-\p\,)|\leq R$ and 
\ $|D_x\phi_0(\x,\t)+\gth \q-(\gth-1)\p|\leq L. $
This completes the proof. 
\end{proof}

\begin{proof}[Outline of proof of Theorem \ref{thm:main} under (A6)$_{\tny-}$] 
%As in the proof under (A6)$_{\tny+}$, 
%we only need to prove the claim under the assumption that  $u_0\in\Lip(\T^n)$. 
%
%once we prove the uniform convergence claim of 
%our theorem for $u_0\in\Lip(\T^n)$, then the claim in the general case is a straightforward 
%consequence. We thus assume henceforth that 

Using Proposition \ref{prop:minus} 
and arguing as the proof of Lemma \ref{lem:w-infty}, we deduce that 
\[
\lim_{t\to\infty}\max\{w(x,t),0\}=0 
\ \ \ \text{ uniformly in } \ \R^n. 
\] 
We fix any $\gep>0$ and choose a constant $T_\gep\equiv T_{\gep,\eta,\gth}>T$ 
so that for any $t\geq T_\gep$, 
\[
w(x,t)\leq \gep \ \ \ \text{ for all } \ x\in\R^n. 
\]
Let $t\geq T_\gep$ and $x\in\R^n$. For any $0\leq s\leq t$, we have
\begin{align*}
u(x,t)-v_0(x)\leq\gep+u(x,s)-v_0(x)+(\gth-1)C_0+\gth\eta(t-s).  
\end{align*}
We may assume that $T_\gep>1$, and from the above, for any $0\leq s\leq 1$, we have  
\begin{equation}\label{2.2.5}
u(x,t)\leq u(x,t-s)+(\gth-1)C_0+\gth\eta+\gep. 
\end{equation}

Since $u\in\Lip(\T^n\tim(0,\,\infty))$ and it is bounded in $\olQ$, the Ascoli-Arzela theorem 
assures that there is a sequence $\{\tau_j\}_{j\in\N}\subset (0,\,\infty)$ 
diverging to infinity such that for some function 
$z\in\Lip(\T^n\tim \R)$,  
\[
\lim u(x,t+\tau_j)=z(x,t) \ \ \ \text{ locally uniformly in } \ \R^{n+1}.
\] 
We see immediately from \erf{2.2.5} that the function $z(x,t)$ is 
nonincreasing in $t$ for every $x$. Furthermore, we infer that for some 
function $u_\infty\in C(\T^n)$,  
\[
\lim_{t\to\infty}z(x,t)=u_\infty(x) \ \ \ \text{ uniformly in }\  \R^n.  
\]
As exactly under (A6)$_{\tny+}$, we deduce from this that  
\[
\lim_{t\to\infty} u(x,t)=u_\infty(x) \ \ \ \text{ uniformly in } \ \R^n, 
\] 
which completes the proof. 
\end{proof}

\section{Conditions (A6)$_{\tny\pm}$}

First of all we restate the conditions (A6)$_{\tny\pm}$ in \cite{BIM} as (A)$_{\tny\pm}$:  

\begin{itemize}
\item[{\rm(A)$_{\tny+}$}] 
There exists $\eta_{0}>0$ such that, 
for any $\eta\in(0,\eta_{0})$, 
there exists $\nu=\nu(\eta)>0$ such that for all $x,p,q\in\R^n$ and $\gth>1$,  
if $H(x,q)\ge\eta$ and $H(x,p)\le0$, 
then 
\[
H(x,p+\gth(q-p))\ge 
\gth H(x,q)+\nu(\gth-1).  
\]
\item[{\rm(A)$_{\tny-}$}] 
There exists $\eta_{0}>0$ such that, 
for any $\eta\in(0,\eta_{0})$,
there exists $\nu=\nu(\eta)>0$ such that for all $x,p,q\in\R^{n}$ and $\gl\in[0,\,1]$, 
if $H(x,q)\le-\eta$ and $H(x,p)\le0$, then  
\[
H(x,(1-\gl)p+\gl q)\le 
\gl H(x,q)-\nu\gl(1-\gl). 
\]
\end{itemize}

Conditions (A6)$_{\tny\pm}$ and (A)$_{\tny\pm}$ 
can be considered as a sort of strict convexity requirements on the function 
$H(x,p)$ in $p$ near the points where $H$ vanishes ((A6)$_{\tny+}$ and (A)$_{\tny+}$ 
are the ones for those points $(x,p)$ where $H(x,p)\geq 0$, while (A6)$_{\tny-}$ and (A)$_{\tny-}$ 
are for those points where $H\leq 0$).  

The condition (H4) in \cite{BS} 
has a general feature more  
than (A)$_{\tny+}$ above, and its additional generality is in the point that includes the key 
assumption in Namah-Roquejoffre \cite{NR}. If we push this point aside, then the condition 
(H4) in \cite{BS} is same as (A)$_{\tny+}$ above. 
 
Now, we give comparison between (A6)$_{\tny+}$ and (A)$_{\tny+}$. 
Let $\eta_0$, $\gth_0$ and $\psi(\eta,\gth)$ be the positive constants from (A6)$_{\tny+}$. 
Note that the key inequality in (A6)$_{\tny+}$ holds with 
$\psi(\eta,\gth)$ replaced by $\min\{\psi(\eta,\gth),\, 1\}$. 
Thus, the behavior of the function $H$ where the value of $H$ is large 
(larger than $\eta_0\gth_0 +1$), 
is irrelevant to condition (A6)$_{\tny+}$, while (A)$_{\tny+}$ requires a certain growth 
of the function $H$ where its value is positive. The function $H$ on $\R^n$ (see Fig. 1 below)
given by  
\[
H(p)=\max\{\min\{|p|^2,\,1\},\,|p|^2/4\}
\]
satisfies (A2)--(A5) and (A6)$_{\tny+}$, as is easily checked. 
However, if $p=0$, $|q|=1$ and $1<\gth<2$, then 
we have 
\[
H(p+\gth(q-p))=H(\gth q)=1<\gth=\gth H(q). 
\]
Therefore, (A)$_{\tny+}$ does not hold with this Hamiltonian $H(x,p)=H(p)$.  
\begin{center}
\includegraphics[height=58mm]{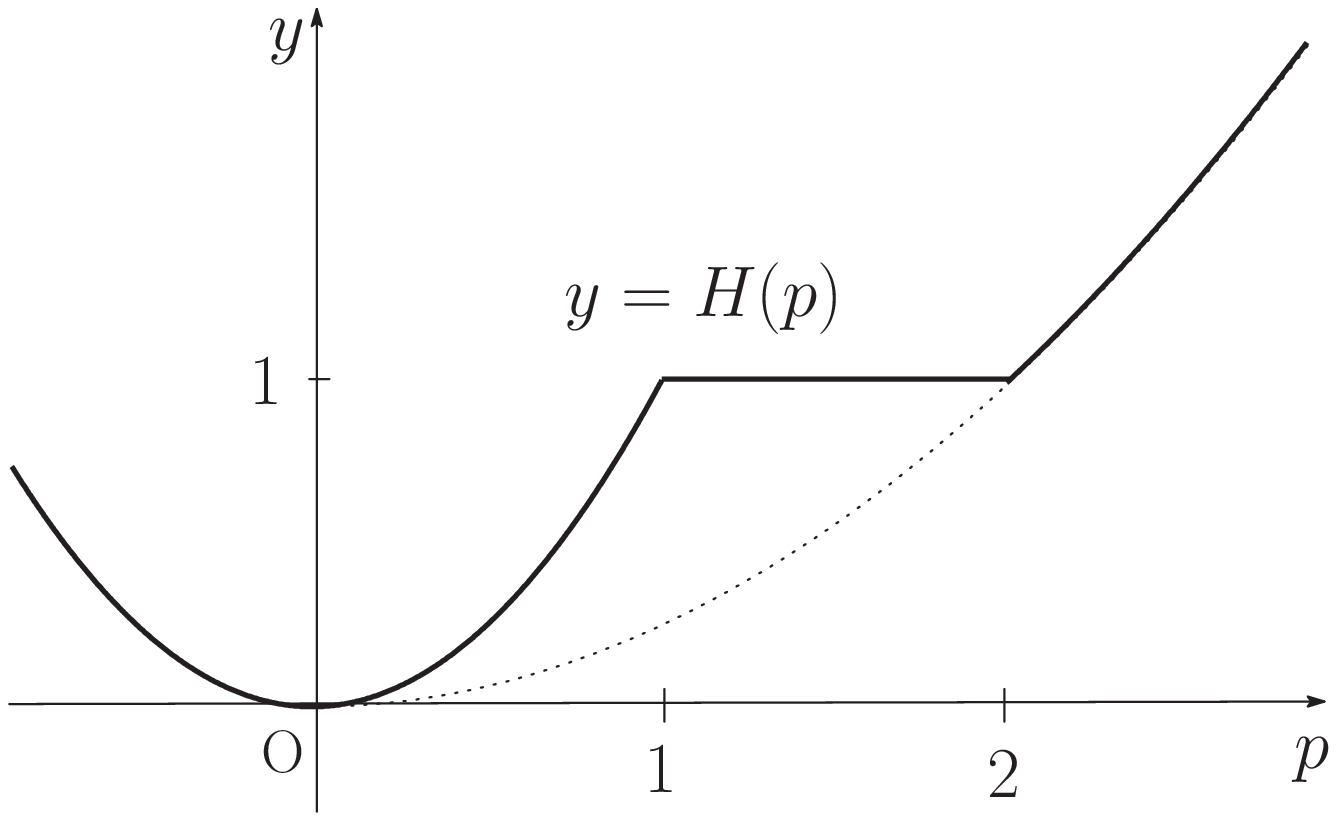} \\\vskip-20pt
{\small Fig. 1}
\end{center} %\vskip20pt

The difference of two conditions observed above  
is concerned with the behavior of the Hamiltonian 
$H(x,p)$ where $H$ is large.     

The following example shows that (A)$_{\tny+}$ is a stronger requirement 
on $H$ than (A6)$_{\tny+}$ even in a neighborhood of the points $(x,p)$ 
where $H$ vanishes. In this regard, the difference between two conditions is that 
the term $\psi(\eta,\gth)$ in (A6)$_{\tny+}$ 
depends generally on $\eta,\,\gth$ while the 
term $\nu(\eta)(\gth-1)$ in (A)$_{\tny+}$ depends linearly in $\gth-1$.   

We define the function $H_0$ (see Fig. 2 below) and $H$ in $C(\R)$ by  
\[
H_0(p)=
\begin{cases}
0 & \text{ for } \ p\leq 0,\\
p+(p-1)^2 &\text{ if } \ p\geq 1, \\
p/2+2(p-1/2)^2 &\text{ if } \ 1/2\leq p\leq 1, \\
\qquad\vdots & \qquad\vdots \\
p/2^{j+1}+2^{j+1}\left(p-1/2^{j+1}\right)^2&\text{ if } \ 1/2^{j+1}\leq p\leq 1/2^j, \\
\qquad\vdots & \qquad \vdots
\end{cases}
\] 
and 
\[
H(p)=|p+1|-1+H_0(p)+H_0(-p-1)
\]

\begin{center}
\includegraphics[height=58mm]{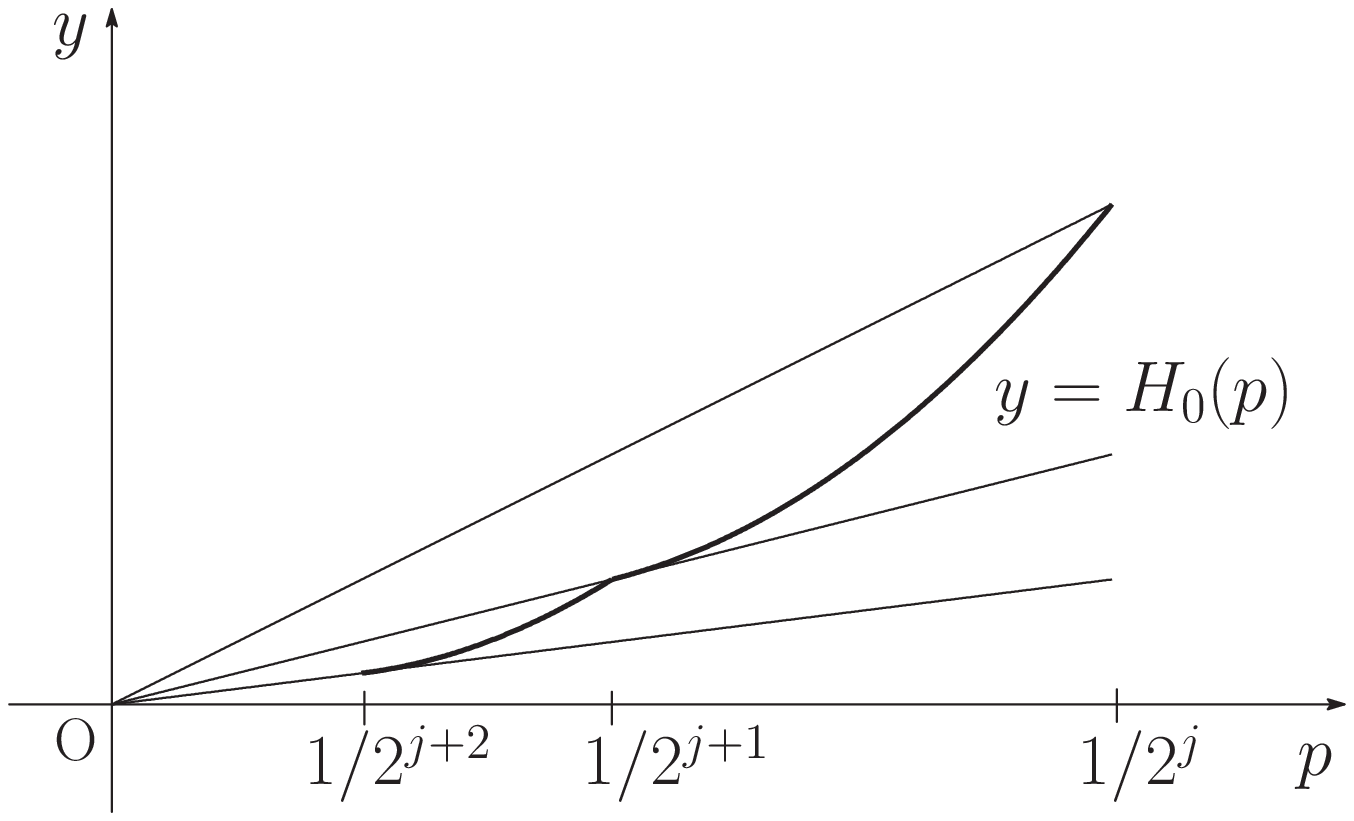}\\ \vskip-16pt
{\small Fig. 2}
\end{center}%\vskip20pt

This Hamiltonian $H$ satisfies (A2)--(A4), and the problem 
\[
H(u'(x))=0 \ \ \ \ \text{ in } \ \R \qquad\text{ and }\qquad u\in C(\T),
\]
where $u'=\d u/\d x$, has a solution $u(x)\equiv 0$. Thus, (A5) is satisfied with our function $H$.  
Moreover, it is easily seen that $H$ satisfies (A6)$_{\tny +}$.  
However, $H$ does not satisfy condition (A)$_{\tny+}$.  
To check this, fix any $j\in\N$ and choose $p=0$ and $q=1/2^{j+1}$. Note that 
\[
H(q)=\fr{1}{2^{j+1}}+\fr 1{2^{2j+2}}, 
\] 
and that for any $\gth\in(1,\,2)$, we have $1/2^{j+1}<\gth q<1/2^j$ and 
\[
H(\gth q)=\fr{\gth}{2^{j+1}}+\fr{\gth}{2^{2j+2}}
+2^{j+1}\left(\fr{\gth}{2^{j+1}}-\fr{1}{2^{j+1}}\right)^2
=\gth H(q)+\fr{(\gth-1)^2}{2^{j+1}}.
\]
Hence, 
\[
H(\gth q)-\gth H(q)=o(\gth-1) \ \ \ \text{ as } \ \gth\to 1+,
\]
which violates the inequality in (A)$_{\tny+}$. Note finally that 
$q=1/2^{j+1}$ can be taken as close to $p=0$ as we wish.

Next, we show that if  $H\in C(\T^n\tim\R^n)$ satisfies (A)$_{\tny-}$, then it  
satisfies (A6)$_{\tny-}$. 

For this, let $H\in C(\T^n\tim\R^n)$ satisfy (A6)$_{\tny-}$. 
Let $\eta_0>0$ be the constant and $\nu$ the function on $(0,\,\eta_0)$ given by 
(A)$_{\tny-}$. 

Fix any $\eta\in(0,\,\eta_0)$ and $\gth>1$, and set $\gl=\gth^{-1}\in(0,\,1)$.  
Let $x,p,q\in\R^n$ and assume that $H(x,p)\leq 0$ and $H(x,q)\geq -\eta$. 
Set 
\begin{equation}\label{3.1}
\psi=\psi(\eta,\gth):=(\gth-1)\min\{\gth^{-1}\nu(\eta),\,\eta\}
=\min\{(1-\gl)\nu(\eta),\,(\gth-1)\eta\}. 
\end{equation}

It is enough to show that 
\begin{equation}\label{3.2}
H(x,p+\gth(q-p))\geq -\gth \eta+\psi. 
\end{equation}
To the contrary, we suppose that 
\begin{equation}\label{3.3}
H(x,p+\gth(q-p))< -\gth \eta+\psi. 
\end{equation}
Set $r=p+\gth(q-p)$ and note that $q=\gl r+(1-\gl) p$. 
Note by the choice of $\psi$ that
\[
H(x,r)< -\gth\eta+(\gth-1)\eta=-\eta. 
\]
Hence, using (A)$_{\tny-}$, \erf{3.2} and \erf{3.1}, we deduce that
\[
H(x,q)=H(x,\gl r+(1-\gl)p)\leq \gl H(x,r)-\nu(\eta)\gl(1-\gl)
<\gl(-\gth+\psi)-\psi\gl=-\eta.
\]
This is a contradiction, which shows that \erf{3.2} holds.

Now, let $H\in C(\T^n\tim\R^n)$ satisfy (A6)$_{\tny+}$, and we show that for each $x\in\R^n$ 
the sublevel set $\{p\in\R^n\mid H(x,p)\leq 0\}$ is convex. 

To do this, we fix any $x\in\R^n$ and let 
$p_1,p_2\in K:=\{p\in\R^n\mid H(x,p)\leq 0\}$. 
We need to show that 
\begin{equation}\label{3.4}
\gl p_1+(1-\gl)p_2\in K \ \ \ \text{ for all }\gl\in[0,\,1].
\end{equation}
We suppose that this is not the case and will get a contradiction.  

Let $\eta_0>0$ and $\gth_0>0$ be the constants from 
(A6)$_{\tny+}$. 
Then, setting
\[
\gl_0=\sup\{\gl\in[0,\,1]\mid \gl p_1+(1-\gl)p_2\not\in K\}, 
\]
we have 
\[
\gl_0 p_1+(1-\gl_0)p_2\in K \ \ \ \text{ by the continuity of \ $H$.}
\]
By the definition of $\gl_0$, we may select a $\gl\in(0,\,\gl_0)$ so that 
\[
\gl p_1+(1-\gl)p_2\not\in K \ \ \ \text{ and } \ \ \ \gl\gth_0>\gl_0. 
\] 
Set 
\begin{align*}
q:&=\gl p_1+(1-\gl)p_2=p_2+\gl(p_1-p_2), 
\\\gth:&=\gl_0/\gl\in (1,\,\gth_0),
\end{align*}
and note that $H(x,q)>0$. Fix an $0<\eta<\eta_0$ so that $H(x,q)\geq \eta$, 
and use condition (A6)$_{\tny+}$, to get 
\[
H(x,p_2+\gth(q-p_2))>\gth\eta>0,
\]
and moreover, 
\[
0<H(x,p_2+\gth(q-p_2))=H(x,\gl_0p_1+(1-\gl_0)p_2)\leq 0.  
\]
This is a contradiction.   

An argument similar to the above guarantees that if $H\in C(\T^n\tim\R^n)$ satisfies 
(A6)$_{\tny-}$, then the sublevel set $\{p\in\R^n\mid H(x,p)<0\}$ is convex for every $x\in\R^n$. 
We leave it for the interested reader to check this convexity property. 

The following example of $H(x,p)=H(p)$ explicitly shows that 
condition (A)$_{\tny-}$ is more stringent than (A6)$_{\tny-}$. 
Define the functions $f, g\in C(\R)$ by
\[
f(p)=
\begin{cases}
0&\text{ if } \ p\leq 0 \ \text{ or } \ p\geq 1, \\[3pt]
-p/2&\text{ if } \ 0\leq p\leq 1/2,\\[3pt]
-(p-1)^2&\text{ if } \ 1/2\leq p\leq 1,
\end{cases}
\] 
\[
g(p)=-p+\sum_{k=1}^\infty 2^{-k}f(2^k p),
\]
and then $H\in C(\R)$ (see Fig. 3 below) by 
\[
H(p)=\max\{g(p),\,g(1-p)\}.
\]
\begin{center}
\includegraphics[height=50mm]{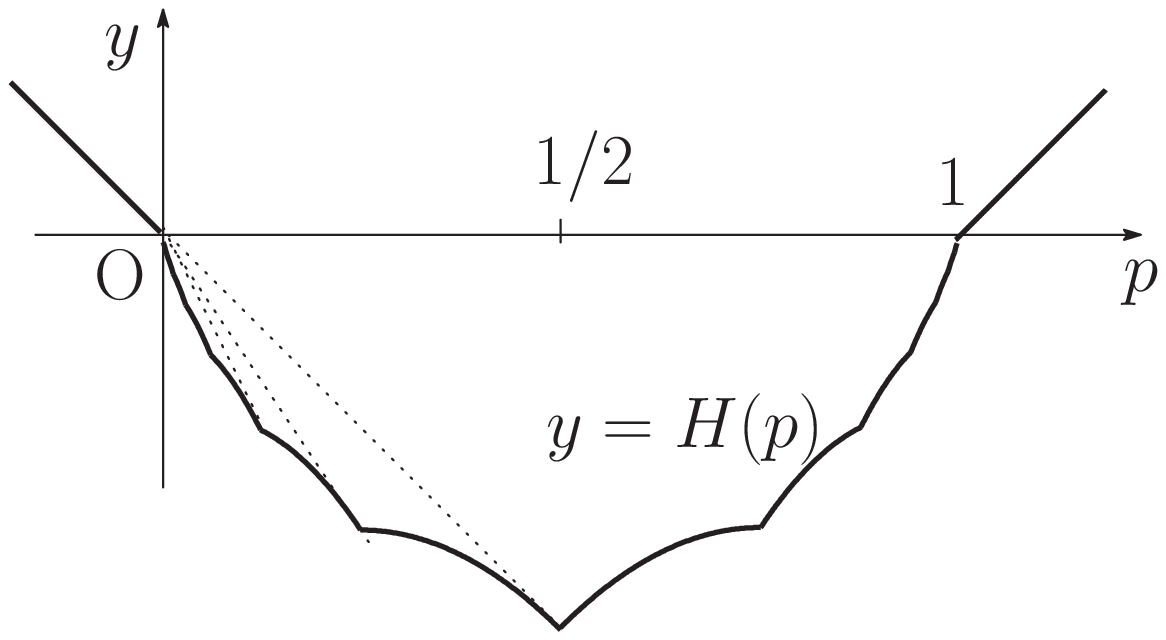}\\ \vskip-8pt 
{\small Fig. 3}
\end{center}%\vskip20pt 
We do not give the detail, but observing that in the $py$ plane,  
for each slope $m<0$, the halh line $y=mp$, $p>0$, meets the graph $y=H(p)$ 
at exactly one point, we can deduce that the function $H$ satisfies (A6)$_{\tny-}$. 
On the other hand, setting $p=0$ and $q=1/2^k$, with $k\in\N$, observing that 
if $\fr 12\leq\gl\leq 1$, then $1/2^{k+1}\leq \gl q\leq 1/2^k$ and that for any $\fr 12\leq\gl\leq 1$, 
\[H(\gl q)=-\fr{k+1}{2^{k+1}}
-\fr{(\gl-1)^2}{2^k}, 
\]
and hence, 
\begin{align*}
H(q)&=-\fr{k+1}{2^{k+1}}, \\
H(\gl q)&=\gl H(q)-\fr{(\gl-1)^2}{2^{k}},
\end{align*}
we may deduce that (A)$_{\tny-}$ does not hold with the current function $H$. 

Next, we remark that 
under hypotheses (A2)--(A4), conditions (A6)$_{\tny+}$ and (A6)$_{\tny-}$ are equivalent 
to the following 
(A7)$_{\tny+}$ and (A7)$_{\tny-}$, respectively.  
\medskip

\begin{itemize}
\item[(A7)$_{\tny+}$]There exist constants $\eta_0>0$ and $\gth_0>1$ such that 
for all $(\eta,\gth)\in (0,\eta_0)\tim(1,\gth_0)$,  
$x, p,q\in\R^n$, if $H(x,p)\leq 0$ and $H(x,q)\geq \eta$, then 
\[
H(x,p+\gth (q-p))> \eta\gth . 
\] 
\item[(A7)$_{\tny-}$]There exist constants $\eta_0>0$ and $\gth_0>1$ and for all  
$(\eta,\gth)\in (0,\eta_0)\tim(1,\gth_0)$, $x, p,q\in\R^n$, if $H(x,p)\leq 0$ and $H(x,q)\geq -\eta$, 
then 
\[
H(x,p+\gth (q-p))>-\eta\gth. 
\] 
\end{itemize}
\medskip

Indeed, it is clear that (A6)$_{\tny\pm}$ imply (A7)$_{\tny\pm}$, respectively. 
On the other hand, assuming that (A7)$_{\tny+}$ holds, choosing $R>0$ so large that
\[
H(x,p)>\eta_0\gth_0 \ \ \ \text{ if } \ |p|>R,
\]
where $\eta_0>0$ and $\gth_0>1$ are the constants from (A7)$_{\tny+}$, 
and setting 
\[
\psi(\eta,\gth)
=\min\{H(x,p+\gth(q-p))-\gth\eta\mid x\in\T^n,\, H(x,p)\leq 0,\, H(x,q)\geq\eta\}
\]
for any $(\eta,\gth)\in(0,\,\eta_0)(1,\,\gth_0)$
we observe that $\psi(\eta,\gth)$ is positive and satisfies
\[
H(x,p+\gth(q-p))\geq \eta\gth+\psi(\eta,\gth)
\]  
for all $(x,p,q)\in\R^{3n}$ such that $H(x,p)\leq 0$ and $H(x,q)\geq\eta$, which shows that 
(A6)$_{\tny+}$ holds. Similarly, we see that (A7)$_{\tny-}$ implies (A6)$_{\tny-}$.

Finally, we remark that under (A2)--(A4), conditions 
(A6)$_{\tny+}$ and (A6)$_{\tny-}$ are equivalent to the following 
(A8)$_{\tny+}$ and (A8)$_{\tny-}$, respectively. 
\smallskip

\begin{itemize}
\item[(A8)$_{\tny+}$]There exist constants $\eta_0>0$ and $\gth_0>1$ and for each 
$(\eta,\gth)\in (0,\eta_0)\tim(1,\gth_0)$ a constant $\psi=\psi(\eta,\gth)>0$ such that 
for all $x, p,q\in\R^n$, if $H(x,p)=0$ and $H(x,q)=\eta$, then 
\[
H(x,p+\gth (q-p))\geq \eta\gth +\psi. 
\] 
\item[(A8)$_{\tny-}$]There exist constants $\eta_0>0$ and $\gth_0>1$ and for each 
$(\eta,\gth)\in (0,\eta_0)\tim(1,\gth_0)$ a constant $\psi=\psi(\eta,\gth)>0$ such that 
for all $x, p,q\in\R^n$, if $H(x,p)=0$ and $H(x,q)=-\eta$, then 
\[
H(x,p+\gth (q-p))\geq -\eta\gth +\psi. 
\] 
\end{itemize}
\smallskip

It is clear that (A6)$_{\tny\pm}$ imply (A8)$_{\tny+}$, respectively.   

We next show that (A8)$_{\tny+}$ implies (A7)$_{\tny+}$, which is equivalent to (A6)$_{\tny+}$. 
We leave it to the reader to check that (A8)$_{\tny-}$ implies (A7)$_{\tny-}$.  

Let $\eta_0$ and $\gth_0$ be 
the constants from (A8)$_{\tny+}$.  We may assume, by replacing $\gth_0$ by 
a smaller one if needed, that $\gth_0<2$. 

Fix any $0<\eta<\eta_0/2$ and $(x,p,q)\in\R^{3n}$ such that $H(x,p)\leq 0$ and 
$H(x,q)\geq\eta$. It is enough to 
show that for all $0<\gth<\gth_0$, 
\begin{equation}\label{3.5}
H(x,p+\gth(q-p)>\gth\eta. 
\end{equation}
We assume for  contradiction that \eqref{3.5} does not hold. We set
\[
\Theta=\{\gth\in(1,\gth_0)\mid H(x,p+\gth(q-p))\leq\gth\eta\}. 
\]
Note by the above assumption that \ $\Theta\not=\emptyset$ and set  
\ $\hat\gth:=\inf\Theta$. It is clear that \ $1\leq \hat\gth<\gth_0$, 
$H(x,p+\hat\gth(q-p))=\hat\gth\eta$ \ since $H(x,q)\geq\eta$ \ 
and \ $H(x,p+\gth(q-p))>\gth\eta$ \ if \ 
$1<\gth<\hat\gth$.

In what follows, we write $H(r):=H(x,r)$ and 
$q_\gth=p+\gth(q-p)$ \ for \ $0\leq\gth<\gth_0$.
We fix a $\gl\in[0,\,1)$ so that \ $H(p+\gl(q-p))=0$. Note that \ $H(q_\gl)=0$. 
\begin{center}
\includegraphics[height=56mm]{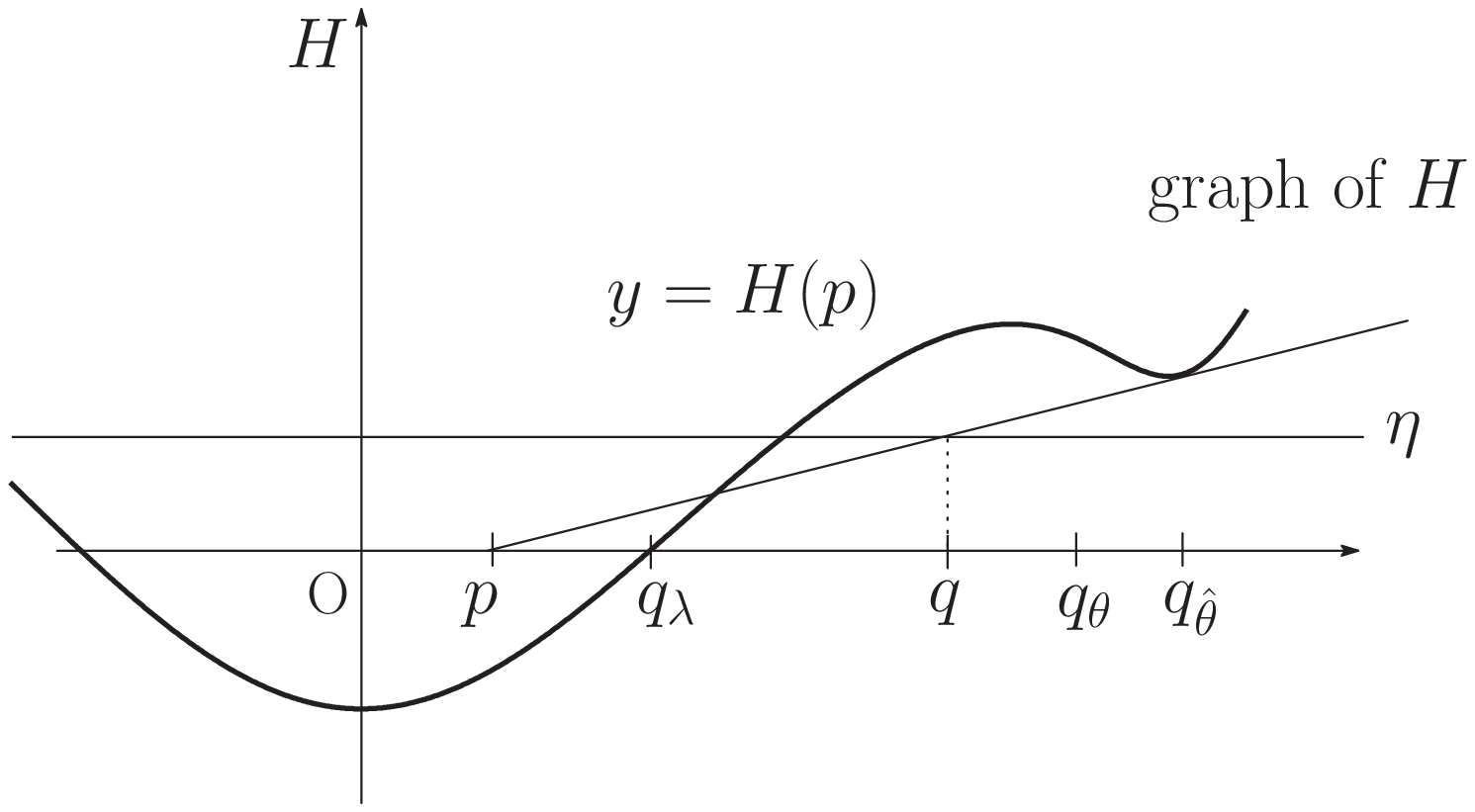}\\ \vskip-8pt 
{\small Fig. 4} 
\end{center}

Consider the case where \ $\hat\gth=1$. In this case we have \ 
$q_{\hat\gth}=q$ \ and \ $H(q)=\eta$. 
By (A8)$_{\tny+}$, we get
\begin{equation}\label{3.6}
H(r+\rho(q-r))>\rho\eta \ \ \ \text{ for all } \ \rho\in(1,\gth_0). 
\end{equation}
Noting that
\[
q_\gl+\rho(q-q_\gl)=p+\gl(q-p)+\rho(q-p-\gl(q-p))=p+(\gl+(1-\gl)\rho)(q-p),
\]
from \eqref{3.6} we get 
\[
H(p+(\gl+(1-\gl)\rho)(q-p))>\rho\eta\geq (\gl+(1-\gl)\rho)\eta \ \ \ 
\text{ for all } \ \rho\in(1,\gth_0), 
\]
which implies that \ $\Theta\cap(1,\,\gl+(1-\gl)\gth_0)=\emptyset$. 
This ensures that \ $
\hat\gth\geq \gl+(1-\gl)\gth_0>1$, 
which contradicts that \ $\hat\gth=1$. 

Consider next the case where $\hat\gth>1$. Recall that 
\ $H(q_{\hat\gth})=\hat\gth\eta<2\eta<\eta_0$ \ and \  
$H(q_\gth)>\gth\eta$ \ for all \ $\gth\in(1,\,\hat\gth)$. 
Setting \ $\eta_\gth:=H(q_\gth)$, we observe that 
if \ $1<\gth<\hat\gth$ \ is close to \ $\hat\gth$, then 
\ $\gth\eta<\eta_\gth<\eta_0$. 
For any such $\gth$, by (A8)$_{\tny+}$, we get
\begin{equation}\label{3.7}
H(q_\gl+\rho(q_\gth-q_\gl))>\rho\eta_\gth \ \  \ \text{ for all } \ \rho\in(1,\,\gth_0).
\end{equation}
Note that \ $  
%\begin{align*} 
q_\gl+\rho(q_\gth-q_\gl)
%&=p+\gl(q-p)
%+\rho(
%p+\gth(q-p) -p-\gl(q-p))
%\\&
=p+(\gl+\rho(\gth-\gl))(q-p)$. 
%\end{align*}
We select $\hat\rho$ so that \ $\hat\gth=\gl+\hat\rho(\gth-\gl)$ or, equivalently, 
$\hat\rho=(\hat\gth-\gl)/(\gth-\gl)$. 
Since $\gth$ is assumed to be close enough to $\hat\gth$, we may assume that 
\ $\hat\rho\in(1,\,\gth_0)$. 
Thanks to \eqref{3.7}, we get  
\[
\hat\gth\eta=H(q_{\hat\gth})
=H(q_\gl+\hat\rho(q_\gth-q_\gl))>\hat\rho\eta_\gth>\fr{\hat\gth-\gl}{\gth-\gl}\eta\gth.
\]
Thus, we get
\ $\hat\gth(\gth-\gl)>\gth(\hat\gth-\gl)$ 
or, equivalently, 
\ $
\gl(\hat\gth-\gth)<0$. 
This is a contradiction.   
We thus see that (A8)$_{\tny+}$ implies (A7)$_{\tny+}$.

\section{A generalization of (A6)$_{\tny\pm}$} \label{sec:NR}

We recall that 
the following conditions on the Hamiltonian $H\in C(\T^n\tim\R^n)$ has been 
introduced by Namah-Roquejoffre \cite{NR} in their study of 
the large time asymptotic behavior of solutions of (CP).  
\begin{itemize}
\item[(NR1)] The function $H(x,p)$ is convex in $p\in\R^n$ for every $x\in\R^n$. \vskip3pt
\item[(NR2)] $\disp \min_{p\in\R^n}H(x,p)=H(x,0)$ \ for all \ $x\in\R^n$.\vskip3pt
\item[(NR3)] $\disp\max_{x\in\R^n}H(x,0)=0$. \vskip3pt 
\item[(NR4)] $\disp \lim_{r\to\infty}\inf\{H(x,p)\mid (x,p)\in\T^n\tim\R^n,\, |p|\geq r\}=\infty$. 
\end{itemize}

Assume for the moment that $H\in C(\T^n\tim\R^n)$ satisfies (NR3). Then the function $v(x)\equiv 0$ 
solves in the classical sense
\[
H(x,Dv(x))=H(x,0) \ \ \ \text{ in }\ \R^n.
\]
Here, if $H(x,0)<0$ for some points $x$, then $v$ is a ``strict'' subsolution of 
$H(x,Du)=0$ in the set $\{x\in\R^n\mid H(x,0)<0\}$. 

We take this observation into account and modify 
conditions (A6)$_{\tny\pm}$ as follows. The new conditions depend 
on our choice of a subsolution $v_0$ of \eqref{S}, 
which plays the same role as the function $v_0$ 
in the proof of Theorem \ref{thm:main}. As we have already 
noted in Remark \ref{rem:v_0}, the function $v_0$ in the proof of Theorem \ref{thm:main} 
is needed to be just a subsolution of \eqref{S} and the outcome may depend on our 
choice of $v_0$. 
Now we fix a subsolution 
$v_0\in C(\T^n)$ of \eqref{S} and choose a nonnegative function $f\in C(\T^n)$ so that 
$v_0$ is a subsolution of 
\[
H(x,Dv_0(x))\leq -f(x) \ \ \ \text{ in } \ \R^n. 
\]
\smallskip

\begin{itemize}
\item[(A9)$_{\tny+}$] 
There exist constants $\eta_0>0$ and $\gth_0>1$ and for each 
$(\eta,\gth)\in (0,\eta_0)\tim(1,\gth_0)$ a constant $\psi=\psi(\eta,\gth)>0$ such that 
for all $x, p,q\in\R^n$, if $H(x,p)\leq -f(x)$ and $H(x,q)\geq \eta$, then 
\[
H(x,p+\gth (q-p))\geq \eta\gth +\psi. 
\] 
\item[(A9)$_{\tny-}$] 
There exist constants $\eta_0>0$ and $\gth_0>1$ and for each 
$(\eta,\gth)\in (0,\eta_0)\tim(1,\gth_0)$ a constant $\psi=\psi(\eta,\gth)>0$ such that 
for all $x, p,q\in\R^n$, if $H(x,p)\leq -f(x)$ and $H(x,q)\geq -\eta$, then 
\[
H(x,p+\gth (q-p))\geq -\eta\gth +\psi. 
\] 
\end{itemize}
\smallskip

The same proof as that of Theorem \ref{thm:main} yields the following proposition. 
We do not repeat its proof here, and leave it to the reader to check the detail. 

\begin{thm}\label{thm:general} 
The assertion of Theorem \ref{thm:main}, with \emph{(A9)$_{\tny\pm}$} in place of 
\emph{(A6)$_{\tny\pm}$}, holds.   
\end{thm}

In the following, we show that if $H\in C(\T^n\tim\R^n)$ satisfies (NR1)--(NR3), then 
(A9)$_{\tny-}$ holds. 

We choose $v_0$ to be the function $v_0(x)\equiv 0$. This function $v_0$ satisfies
\[
H(x,Dv_0(x))=H(x,0)=-f(x) \ \ \ \text{ for all } \ x\in\R^n,  
\]
where $f(x):=-H(x,0)$. 

Fix any $x,p,q\in\R^{3n},\,\eta>0$ such that $H(x,p)\leq -f(x)$ and $H(x,q)\geq -\eta$. 
To prove that (A9)$_{\tny-}$ holds with $f(x)=-H(x,0)$,  
it is enough to show that there is a constant $\psi(\eta,\gth)>0$ such that
\[
H(x,p+\gth(q-p))\geq -\gth\eta+\psi(\eta,\gth). 
\]

Since 
\[H(x,p)\leq -f(x)=H(x,0)=\min_{r\in\R^n}H(x,r),\] 
we have \ $H(x,p)=-f(x)=H(x,0)$. 
Fix any $\gth>1$. By the convexity of $H$, we have 
\begin{align*}
H(x,p+\gth(q-p))&\geq H(x,p)+\gth(H(x,q)-H(x,p))
\\&=-f(x)+\gth(-\eta+f(x))
=-\gth\eta+(\gth-1)f(x),
\end{align*}
while we have 
\[
H(x,p+\gth(q-p))\geq H(x,0)=-f(x)=-\gth\eta+(\gth\eta-f(x)). 
\]
Setting
\[
\psi(\eta,\gth)=\min_{x\in\T^n}\max\{(\gth-1)f(x),\,\gth\eta-f(x)\},
\]
we observe that \ $\psi(\eta,\gth)>0$ \ and   
\[
H(x,p+\gth(q-p))\geq -\gth\eta+\psi(\eta,\gth). 
\]
Thus, $H$ satisfies (A9)$_{\tny-}$.

\bibliographystyle{amsplain}

\end{document}